\documentclass[11pt,reqno]{article}
\usepackage{amsfonts}
\usepackage{amsmath}
\usepackage{amssymb}
\usepackage{epsfig}
\usepackage{amsthm}
\usepackage[english]{babel}
\usepackage{hyperref}
\newcommand{\N}{\mathbb N}

\newcommand{\R}{\mathbb R}

\newcommand{\E}{\mathbb E}
\newcommand{\pr}{\mathbb P}

\newcommand{\half}{\frac{1}{2}}
\newcommand{\salj}{\mathcal{F}}

\newcommand{\Var}{\mathop{\rm Var}\nolimits}
\newenvironment{centre}{\begin{center}}{\end{center}}

\begin{document}
\theoremstyle{plain}
\newtheorem{thm}{Theorem}
\newtheorem{lem}[thm]{Lemma}
\newtheorem{cor}[thm]{Corollary}
\newtheorem{prop}[thm]{Proposition}

\title{The connected component of the partial duplication graph}
\author{Jonathan Jordan}
\maketitle

\begin{abstract}We consider the connected component of the partial duplication model for a random graph, a model which was introduced by Bhan, Galas and Dewey as a model for gene expression networks.  The most rigorous results are due to Hermann and Pfaffelhuber, who show a phase transition between a subcritical case where in the limit almost all vertices are isolated and a supercritical case where the proportion of the vertices which are connected is bounded away from zero.

We study the connected component in the subcritical case, and show that, when the duplication parameter $p<e^{-1}$, the degree distribution of the connected component has a limit, which we can describe in terms of the stationary distribution of a certain Markov chain and which follows an approximately power law tail, with the power law index predicted by Ispolatov, Krapivsky and Yuryev.  Our methods involve analysing the quasi-stationary distribution of a certain continuous time Markov chain associated with the evolution of the graph.\end{abstract}

\section{Introduction}

The partial duplication model is a model for a growing random graph introduced by Bhan, Galas and Dewey \cite{bhan} as a model for gene expression networks, and further studied by Chung, Lu, Dewey and Galas \cite{chungetal}, Bebek et al \cite{cooperdup}, Ispolatov, Krapivsky and Yuryev \cite{ispolatov}, modelling protein-protein interaction networks, Li, Choi and Wu \cite{lcw} and Hermann and Pfaffelhuber \cite{pfaffel}.  The model is that the graph evolves in discrete time and that at each time point, a single vertex is chosen uniformly at random to ``duplicate''.  This means that a new vertex, which we can think of as an offspring or mutant of the chosen vertex, is added to the graph, and is connected to the neighbours of the chosen vertex, each with probability $p$ (independently of each other) where $p\in (0,1]$ is a parameter of the model.  Note that in our model the new vertex is not connected to the vertex it was duplicated from.  The case where $p=1$ is referred to as \emph{full duplication} and has some special properties, while the cases where $p<1$ are referred to as \emph{partial duplication}.

In this model, it is clear that if a vertex $v$ has degree zero then it will continue to do so for all time, and furthermore that any vertex duplicated from $v$ will also have degree zero.  This suggests the possibility that if $p$ is small enough then in the limit almost all vertices will have degree zero.  Hermann and Pfaffelhuber \cite{pfaffel} show that this situation occurs if $p\leq p_c$, where $p_c$ is the unique root of $pe^p=1$, while if $p>p_c$ there is no non-defective limiting degree distribution.  They also obtain a number of results concerning the asymptotics of the numbers of cliques and stars of different sizes in the graph.

In the case where almost all vertices have degree zero, a natural question is to consider the degree distribution of the connected component of the graph, assuming that the initial graph is connected.  This was explored by Ispolatov, Krapivsky and Yuryev \cite{ispolatov} using non-rigorous methods, suggesting a power-law distribution for the degrees with index given by the solution to $-3+\beta+p^{\beta-2}=0$ when $p<e^{-1}$, and index $-2$ when $e^{-1}\leq p<\frac12$; it is also considered in Section 2 of Hermann and Pfaffelhuber \cite{pfaffel}, where the conjecture that the connected component satisfies a power law degree distribution is mentioned.

The aim of this paper is to discuss the behaviour of degrees in this connected component in more detail, using a method involving a quasi-stationary distribution of a certain continuous time Markov chain.  We will show that, for $p<e^{-1}$, the expected number of vertices of a particular degree, when normalised appropriately, converges to a non-degenerate limit and that the degree distribution of the connected component converges in probability to this distribution.  We can describe this limit in terms of the stationary distribution of a related Markov chain, and we will also show that this distribution has tail behaviour close to that of a power law of the index suggested in \cite{ispolatov}.  Our proofs have some similarity with the discrete time Markov chain methods used in Jordan \cite{randrep} for a different model.

It is observed non-rigorously in \cite{ispolatov} that considering the behaviour of the connected component and letting $p\to 0$ gives the preferential attachment mechanism of Barab\'{a}si and Albert \cite{scalefree1999}, and as is well-known (first rigorously proved by Bollob\'{a}s, Riordan, Spencer and Tusn\'{a}dy \cite{brst}) that model gives a degree distribution which is asympotically a power law with tail index $-3$.  We will see that the tail indices of the distributions in our model converge to $-3$ as $p\to 0$.

As an illustration of the sort of graphs which the model produces and how the density of edges increases with $p$, simulations of the model with $500$ vertices in the connected component and three values of $p$, each starting from a ring of five vertices, are displayed in Figure \ref{fig}.
\begin{figure}
\begin{centre}
\begin{tabular}{ccc}\scalebox{0.35}{{\includegraphics{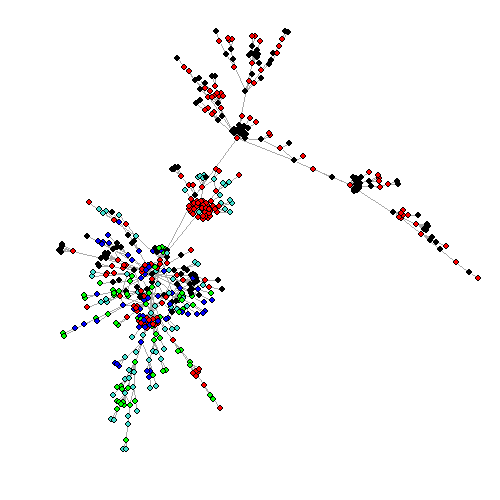}}} & \scalebox{0.35}{{\includegraphics{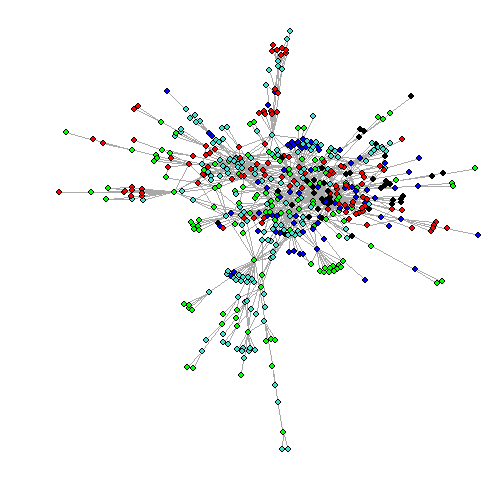}}} & \scalebox{0.35}{{\includegraphics{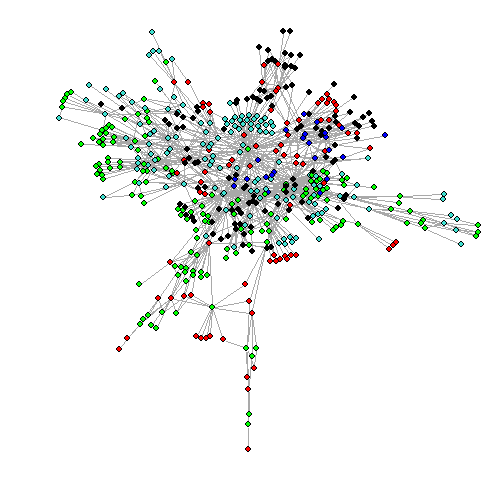}}}\end{tabular}
\caption{Example simulations of the connected component, with $500$ vertices and $p=0.05$, $0.2$ and $0.36$ (left to right), each starting from a ring of five vertices.  The colours of the vertices reflect which of the five ancestor vertices they are descended from. \label{fig}}
\end{centre}
\end{figure}

\subsection{Other duplication models}

Although the growth to $1$ of the proportion of degree zero vertices can be seen as a natural feature of the model, with these vertices reflecting unsuccessful mutants which have lost all their interactions, we note that there are also variants of the duplication model which avoid it.  One idea is for the new vertex to additionally connect to vertices which were not neighbours of its parent with some small probability; this is considered by Pastor-Satorras, Smith and Sol\'{e} \cite{pastorsatorras}, and is also studied by Bebek et al \cite{Bebek2006}, Kim, Krapivsky, Kahng and Redner \cite{kimdup} and Raval \cite{raval}.  These extra edges can be seen as due to mutations causing the new vertex to interact with vertices which its parent did not.

Another idea, which is considered in chapter 4 of Chung and Lu \cite{chunglu}, is to always maintain a connected graph (assuming that the initial graph is connected) by the new vertex always connecting to the vertex it was duplicated from.  This model appears to have been rediscovered by Li, Chen, Cheng and Wang \cite{li2015analysis} where it is suggested as a model for social networks, a context where the connection to the parent vertex is natural. The results in \cite{chunglu} suggest that for $p<p_c$ (the same $p_c$ as for our model) the expected degree distribution converges to a limit which has a power law type tail, with index depending on $p$, but not the same index as in our results.

A different family of duplication graph models is introduced by Backhausz and M\'{o}ri \cite{backhausz1}, and extended by Th\"{o}rnblad \cite{thoernblad}.  In the models of \cite{backhausz1}, two vertices are selected at each time step.  One is duplicated with full duplication, so that all its edges become edges of the new vertex, and one has its edges deleted (but is not deleted itself).  For these models, \cite{backhausz1} shows almost sure convergence to a particular degree distribution, which has a stretched exponential tail.  In the extended model of Th\"{o}rnblad \cite{thoernblad}, also studied by Backhausz and M\'{o}ri \cite{backhausz2}, at each time step a single vertex is chosen, and duplicated with probability $\theta$ and its edges deleted with probability $1-\theta$.  For this model, \cite{thoernblad} shows almost sure convergence to a degree distribution which has a phase transition from exponential to power law decay at $\theta=1/2$.  At $\theta=1/2$ itself the behaviour is like that of the model in \cite{backhausz1}.  The analyses in these papers rely on the clique structure of the graph, which is associated with the full duplication.

We briefly mention two more extensions.  Hamdi, Krishnamurthy and Yin \cite{hamdi}, also motivated by social networks, introduce a variant where the probabilities that a vertex is deleted and that when a duplication step takes place that the new vertex connects to each neighbour of its parent are dependent on the state of an underlying Markov chain.  Finally, a model where the duplication probabilities are proportional to the degree instead of uniform is considered in Cohen, Jordan and Voliotis \cite{prefdup}, but rigorous results are only obtained for the case of full duplication.

\section{Definitions and results}

The model we consider can be defined in discrete time as in Hermann and Pfaffelhuber \cite{pfaffel}.  We define a parameter $p\in (0,1]$.  We start at time $n_0$ with an undirected graph $G_{n_0}$, which has $n_0>1$ vertices, labelled $1,2,\ldots,n_0$, and which we assume to be connected.  For $n \geq n_0$, and given $G_n$, which has $n$ vertices, we form $G_{n+1}$ by picking a random vertex $u$, and adding a new vertex (which we will label as $n+1$) which is connected to each neighbour of $u$ with probability $p$, independently of each other, and to no other vertices.

Let $U_n$ be the degree of a vertex chosen uniformly at random from the graph at time $n$.  We consider the distribution of $U_n$ without conditioning on the graph, and show the following result.

\begin{thm}\label{main}Assume $0<p<e^{-1}$.  \begin{enumerate} \item For each $k\geq 1$, there exists $a_k$ such that $\lim_{n\to\infty}\pr(U_n=k|U_n\neq 0)=a_k$, and furthermore $\lim_{n\to\infty}\pr(U_n=k|G_m,U_n\neq 0)=a_k$ for any $m$. \item The proportion of vertices of the connected component of $G_n$ which have degree $k$ converges to $a_k$ as $n\to\infty$, in probability. \item Let $\beta(p)$ be the solution to $-3+\beta+p^{\beta-2}=0$.  Then the tail behaviour of $a_k$ is close to a power law of index $-\beta(p)$, in the sense that as $k\to\infty$, $a_k/k^{-q}\to 0$ if $q<\beta(p)$ and $a_k/k^{-q}\to \infty$ if $q>\beta(p)$.\end{enumerate}
\end{thm}

In Section \ref{qs}, we will show how to derive the distribution given by the $a_k$ as a quasi-stationary distribution of a certain continuous time Markov chain, and we will use Foster-Lyapunov methods to get indications of the tail behaviour, which will give part (c) of  Theorem \ref{main}.   We will then complete the proof of part (a) in Section \ref{proofmain1}, and the proof of part (b) in Section \ref{proofmain2}.

We will make frequent use of the following embedding of our model in continuous time.  We start at time zero with a fixed connected graph $\Gamma_0$ with $n_0>1$ vertices, and define a continuous time Markov chain $(\Gamma_t)_{t\geq 0}$ on the state space of graphs by saying that each vertex duplicates at times given by a Poisson process of rate 1, independently of everything else, with the rules for the addition of a new vertex when a duplication happens being as before.  We will define $N_t$ to be the number of vertices in $\Gamma_t$, and will maintain the above labelling of the vertices: the vertices of $\Gamma_0$ are labelled $1,2,\ldots,n_0$, and the later vertices are numbered in order of arrival so that the most recent vertex at time $t$ is labelled $N_t$.  We observe that the process $(N_t)_{t\geq 0}$, which gives the number of vertices in the system, follows the well-known Yule process introduced by \cite{yule1925}.  We also note that a different continuous time embedding of the process, with vertices in a graph with $N_t$ vertices duplicating at rate $1+1/N_t$, was used by Hermann and Pfaffelhuber \cite{pfaffel}.

\section{Vertex tracking and the quasi-stationary distribution}\label{qs}

In the continuous time version of our process, we define a \emph{tracked vertex} $(V_t)_{t\geq 0}$ as follows.  We start by choosing $V_0$ uniformly at random from the vertices of $\Gamma_0$, and then say that the process $(V_t)_{t\geq 0}$ will have a jump at time $t$ if and only if the vertex $V_{t-}$ is duplicated at time $t$, in which case it will jump to the new vertex.  Let the degree of $V_t$ be $D_t$; then $(D_t)_{t\geq 0}$ is a continuous time Markov chain on $\N_0$ and from $j$ jumps to $j+1$ when a neighbour of the currently tracked vertex is duplicated and the edge retained (rate $jp$) and to $k<j$ when the currently tracked vertex is duplicated together with $k$ of its edges.  The generator $Q$ of this continuous time Markov chain with state space $\N_0$ is thus given by
\begin{eqnarray*}
q_{j,j+1} &=& jp\\
q_{j,k} &=& \binom{j}{k}p^k(1-p)^{j-k} \text{ for }0\leq k\leq j-1\\
q_{j,j} &=& -(jp+1-p^j).
\end{eqnarray*}
As expected, $0$ is an absorbing state here: if at some time $t$ the tracked vertex has degree zero then this will remain the case at all later times.

We note that the events that $V_t$ is a particular vertex and that the degree of that vertex is $d$ are independent; this is because in the continuous time model the changes in degree of a particular vertex, which are when its neighbours duplicate, are independent of its duplications.

We will be interested in $D_t$, the degree of our tracked vertex, conditional on it not being zero, that is on it being part of the connected component.  To investigate this, we will use the theory of quasi-stationary distributions of Markov chains, for which we will follow Pollett \cite{pollettquasi1988}, which considers quasi-stationary distributions for continuous time Markov chains on countable state spaces.  A quasi-stationary distribution in this context is a left eigenvector of the generator matrix, excluding the row and column corresponding to state 0, which sums to 1 and has all entries non-negative.  The eigenvalue is necessarily negative, and we will write it as $-\lambda$.  Under certain conditions the distribution of the state of the chain conditional on not having hit zero will converge to a quasi-stationary distribution.

A quasi-stationary distribution with eigenvalue $-\lambda$ for a chain with the generator $Q$ will satisfy
\begin{equation}\label{Aqstat}a_{k-1}(k-1)p+
\sum_{j=k+1}^{\infty}a_j\binom{j}{k}p^k(1-p)^{j-k}-a_k(1-\lambda+kp-p^k)=0,\end{equation} for $k\in \N$, from which we obtain $$\sum_{j=k}^{\infty}a_j\binom{j}{k}p^k(1-p)^{j-k}=a_k(1-\lambda+kp)-a_{k-1}(k-1)p.$$

In Section 3 of \cite{pollettquasi1988}, a \emph{$\lambda$-invariant measure} is defined to be a positive left eigenvector $m$ of $Q$ restricted to $\N$ with eigenvalue $-\lambda$, so that a quasi-stationary distribution is a $\lambda$-invariant measure which sums to $1$, and a \emph{$\lambda$-invariant vector} is defined to be a positive right eigenvector $x$ of $Q$ restricted to $\N$ with eigenvalue $\lambda$.

Also in \cite{pollettquasi1988}, given the existence of a $\lambda$-invariant vector and measure, two generator matrices for continuous time Markov chains are defined on (in our context) $\N$.  Given a $\lambda$-invariant measure $m$ for $Q$, the \emph{$\lambda$-reverse} of $Q$ with respect to $m$ is a generator matrix $Q^*$ defined by letting $$q^*_{jk}=m_k(q_{kj}+\lambda\delta_{jk})/m_j,$$ and given a $\lambda$-invariant vector $x$ for $Q$ the \emph{$\lambda$-dual} of $Q$ with respect to $x$ is a generator matrix $\bar{Q}$ defined by letting $$\bar{q}_{jk}=(q_jk+\lambda \delta_{jk})x_k/x_j.$$

The following result suggests that if we are to have a quasi-stationary distribution with finite mean we should expect $\lambda=1-2p$.
\begin{prop}Assume that $a$ is a quasi-stationary distribution of $p$ with eigenvalue $-\lambda$, and that $a$ has a finite mean.  Then $\lambda=1-2p$.\end{prop}
\begin{proof}We follow the ideas in Chapter 4 of Chung and Lu \cite{chunglu} for a related model and work with the generating function of the distribution $a$, $F(z)=\sum_{j=1}^{\infty}a_jz^j$.  Note that (setting $a_0=0$) \begin{eqnarray*} F(pz+1-p) &=& \sum_{k=0}^{\infty}z^k \sum_{j=k}^{\infty}a_j\binom{j}{k}p^k(1-p)^{j-k} \\ &=& \sum_{k=1}^{\infty}z^k \sum_{j=k}^{\infty}a_j\binom{j}{k}p^k(1-p)^{j-k} +\sum_{j=1}^{\infty}a_j (1-p)^j \\ &=& \sum_{k=1}^{\infty}\left(a_k(1-\lambda)z^k+kpa_k z^k-(k-1)a_{k-1}pz^k\right)+F(1-p).\end{eqnarray*}  Hence we get \begin{equation}\label{Aqstatgf}F(pz+1-p)=(1-\lambda)F(z)+p(z-z^2)F'(z)+F(1-p).\end{equation}
Considering $F(1)=1$, this gives $F(1-p)=\lambda$, as we assume a finite mean.  We can also see that $$F'(z)=\frac{F(pz+1-p)-(1-\lambda)F(z)-F(1-p)}{p(z-z^2)},$$ and taking limits as $z\uparrow 1$, again assuming the limit exists, we get $F'(1)=\frac{1}{-p}F'(1)(p-(1-\lambda))$ and hence $\lambda=1-2p$.
\end{proof}

It turns out that in our setting it is easy to identify a $(1-2p)$-invariant vector.
\begin{lem}\label{invariant}Let $p<\half$.  A $(1-2p)$-invariant vector for $Q$ is given by $x_k=k$, and this is unique up to a multiplicative constant.\end{lem}
\begin{proof}The equations for a $(1-2p)$-invariant vector for $Q$ are, for $j\geq 1$, $$jp x_{j+1}-(jp+1)x_j+\sum_{k=1}^j \binom{j}{k} p^k (1-p)^{j-k} x_k=(2p-1)x_j,$$ giving $$jpx_{j+1}=(2+j)px_j-\sum_{k=1}^j \binom{j}{k} p^k (1-p)^{j-k} x_k,$$ and if we set $x_1=r$ then solving the equations inductively gives $x_k=rk$.
\end{proof}

Using Lemma \ref{invariant}, the $(1-2p)$-dual of $Q$, $\bar{Q}$, with respect to $x$ is given by $$\bar{q}_{jk}=\left\{\begin{array}{lr}\binom{j-1}{k-1}p^k(1-p)^{j-k} & 1<k<j \\ (j+1)p & k=j+1 \\ p^j-(2+j)p & k=j.\end{array}\right.$$  We can now use this to identify our quasi-stationary distribution.

\begin{prop}If $\bar{Q}$ defines a positive recurrent Markov chain, then there exists a quasi-stationary distribution $a$ with eigenvalue $-(1-2p)$ for $Q$.\end{prop}

\begin{proof}
First of all, it is clear that in our setting both $\bar{Q}$ and $Q^*$ are irreducible.  From \cite{pollettquasi1988}, both $Q^*$ and $\bar{Q}$ have the same stationary measure, given by $u$ with $u_j=m_jx_j$.  As we know $x_k=k$ for our $Q$ and $\lambda=1-2p$, we can thus get a quasi-stationary distribution of $Q$ with $\lambda=1-2p$ by defining $u$ to be the unique stationary distribution for $\bar{Q}$, letting $m_k=u_k/k$ and then normalising so that $a_j=m_j/\sum_{i=1}^{\infty}m_i$.
\end{proof}

We note that this cannot give a quasi-stationary distribution with an infinite mean, as then $u$ would not give a probability distribution.

This now allows us to use results on convergence to quasi-stationary distributions to show that the distribution of the degree of our tracked vertex converges.
\begin{prop}\label{converge}If $\bar{Q}$ defines a positive recurrent Markov chain, then for any $j$ and $k$ we have that $$\lim_{t\to\infty} \pr(D_t \geq 1|D_0=j)e^{(1-2p)t}=jm_k,$$ and furthermore that $$\pr(D_t=k|D_t\geq 1) \to a_k \text{ as }t \to \infty.$$\end{prop}

\begin{proof}
By Lemma 3.3(a)(ii) of \cite{pollettquasi1988} and using Lemma \ref{invariant}, we have that $$\pr(D_t=k|D_0=j)=\frac{j}{k}e^{-(1-2p)t}\pr(X_t=k|X_0=j),$$ where $(X_t)_{t\geq 0}$ is a continuous time Markov chain with generator $\bar{Q}$, and the first part follows on taking limits as $t\to\infty$ and recalling the definition of $m_k=u_k/k$.  For the second part, $$\pr(D_t \geq 1|D_0=j)e^{(1-2p)t}=j\E\left(\frac{1}{X_t}|X_0=j\right).$$
We can also calculate $$\pr(D_t=k|D_0=j,D_t\geq 1)=\frac{\pr(X_t=k|X_0=j)}{k\E\left(\frac{1}{X_t}|X_0=j\right)}.$$  As we are assuming $(X_t)_{t\geq 0}$ is positive recurrent, $\pr(X_t=k|X_0=j)\to u_k$ as $t\to\infty$, and furthermore $\E\left(\frac{1}{X_t}|X_0=j\right)\to \sum_{i=1}^{\infty}\frac{u_i}{i}=\sum_{i=1}^{\infty}m_i$ as $t\to\infty$.  Hence $$\lim_{t\to\infty}\pr(D_t=k|D_0=j,D_t\geq 1)=\frac{m_k}{\sum_{i=1}^{\infty}m_i}=a_k$$ for any $j$, giving the result.\end{proof}

Our aim now is to find when $\bar{Q}$ is positive recurrent, and to find out more about our quasi-stationary distribution $a$ when it is.  We will do this via a Foster-Lyapunov approach to investigating the tail of a stationary distribution and whether one exists.  Given a test function $V$, the drift at $x$ is given by $\Delta V(x)=\bar{Q} V(x)$, which is \begin{equation}\label{drift}p(x+1)V(x+1)-p(x+2)V(x)+p\E(V(1+Y)),\end{equation} where $Y\sim Bin(x-1,p)$.

\begin{prop}\label{qmoment}Let $q>0$. \begin{enumerate}\item  If $-1+q+p^q<0$ then $\bar{Q}$ is positive recurrent and its stationary distribution has a $q$th moment. \item If $p^q=1-q$, $p<e^{-1}$ and $r>0$ then $\bar{Q}$ is positive recurrent and a random variable $X$ with its stationary distribution has $\E(X^q(\log(X+1))^{-(r+1)})$ finite.\end{enumerate}\end{prop}
\begin{proof}
We apply Theorem 4.2 of Meyn and Tweedie \cite{meyntweedie1993}, which in our setting with state space equal to $\N$ tells us that, given a function $f: \N \to [1,\infty)$, if there exists a function $V: \N \to \R^{+}$ such that $$\Delta V(x) \leq -c_1 f(x)+c_2$$ then the Markov chain is positive recurrent and that a random variable $X$ with its stationary distribution has $\E(f(X))$ finite.

For $f(x)=x^q$, set $V(x)=f(x)=x^q$.  Then \eqref{drift} becomes $$\Delta V(x)= p((x+1)^{q+1}-(x+2)x^q+\E((1+Y)^q)).$$  For large $x$ the concentration of the Binomial around its mean will give $\E((1+Y)^q))\sim (1+(x-1)p)^q$, giving, as $x\to\infty$, \begin{eqnarray*}\Delta V(x) &\sim & p (x+1)^q\left((x+1)-(x+2)\left(\frac{x}{x+1}\right)^q+\left(\frac{1+(x-1)p}{x+1}\right)^q\right)\\ &\sim & p(x+1)^q\left((x+1)-(x+2)\left(1-\frac{q}{x+1}\right)+p^q\right) \\ & \sim & p(x+1)^q(-1+q+p^q).\end{eqnarray*}  Hence, if $-1+q+p^q<0$, then we will have $\Delta V(x) \leq -c_1 f(x)+c_2$ as required, showing that the stationary distribution has a $q$th moment.

Now let $f(x)=(x+1)^q(\log (x+1))^{-(r+1)}$ with $q$ such that $p^q=1-q$, and let $V(x)=x^q(\log (x+1))^{-r}$.  Then, similarly to the above, we get \begin{eqnarray*}\Delta V(x) & \sim & \frac{p(x+1)^q}{(\log (x+2))^{r}} \left((x+1)-(x+2)\left(\frac{x}{x+1}\right)^q\left(\frac{\log (x+2)}{\log (x+1)}\right)^r+\left(\frac{1+(x-1)p}{x+1}\right)^q\left(\frac{\log (x+2)}{\log (2+p(x-1))}\right)^r\right) \\ & \sim & \frac{p(x+1)^q}{(\log (x+2))^{r}}\left(-1+q-r\frac{x+2}{(x+1)\log(x+1)}+p^q\left(\frac{\log (x+2)}{\log (2+p(x-1))}\right)^r\right) \\ & \sim & \frac{p(x+1)^q}{(\log (x+1))^{r}}(1-q)\left(-1-\frac{r}{(1-q)\log(x+1)}+\left(\frac{\log (x+2)}{\log (2+p(x-1))}\right)^r\right) \\ & \sim & \frac{p(x+1)^q(1-q)}{(\log (x+1))^{r}}\left(-1-\frac{r}{(1-q)\log(x+1)}+1-\frac{r \log p}{\log p+\log x}\right) \\ & \sim & \frac{p(x+1)^q(1-q)}{(\log (x+1))^{r}}\frac{r}{\log x}\left(-\log p-\frac{1}{1-q}\right) \\ & \sim & \frac{(x+1)^q}{(\log(x+1))^{r+1}}p(1-q)r\left(-\log p-\frac{1}{1-q}\right). \end{eqnarray*} As $p^q=1-q$, we have $\log p=\frac{\log(1-q)}{q}$ and thus, as $\frac{-\log(1-q)}{q}-\frac{1}{1-q}<0$, we will have $\Delta V(x) \leq -c_1 f(x)+c_2$, and hence that a random variable $X$ with the stationary distribution of $\bar{Q}$ has $\E(f(X))$ finite.
\end{proof}

\begin{cor}\label{prange}If $p<e^{-1}$ then $\bar{Q}$ is positive recurrent.\end{cor}
\begin{proof}This follows from Proposition \ref{qmoment} and the fact that $-1+q+p^q<0$ for sufficiently small $q>0$ if $p<e^{-1}$.
\end{proof}

By similar arguments, we can also obtain some negative results.

\begin{prop}\label{transient}If $p>e^{-1}$ then $\bar{Q}$ is transient.\end{prop}
\begin{proof}
By Theorem 7.2.2 of Menshikov, Popov and Wade \cite{menshikov_popov_wade_2016}, it will be enough to find a threshold $x_0\in \N$ and bounded function $V:\N \to \R^{+}$ such that $\Delta V(x)< 0$ for $x\geq x_0$ and $V(y)<\inf_{x< x_0} V(x)$ for some $y\geq x_0$.

Consider the non-negative bounded test function $V(x)=x^{-q}$, for some $q>0$.  Then $$\Delta V(x)=p((x+1)(x+1)^{-q}-(x+2)x^{-q}+\E((1+Y)^{-q})),$$ and $$\lim_{x\to\infty}x^q \Delta V(x)=-p(1+q-p^{-q}),$$ which is negative for some $q>0$ if and only if $p>e^{-1}$.  Hence, if $p>e^{-1}$ we can choose $q$ so that $\Delta V(x)<0$ for $x$ sufficiently large, which gives the result.\end{proof}

\begin{prop}\label{noqmoment} Let $q>0$ with $-1+q+p^q\geq 0$.  Then the stationary distribution of $\bar{Q}$, if it exists, does not have a $q$th moment.\end{prop}
\begin{proof}
Let $V(x)=x^q$, and first consider the case where $-1+q+p^q> 0$.  Then, as in Proposition \ref{qmoment} here we have $\Delta V(x)\sim p (x+1)^q(-1+q+p^q)$ as $x\to\infty$, but here this is positive.  Hence there exists $x_0$ such that $\E((\max(X_t,x_0))^q)$, if it exists, is strictly increasing in $t$, which means a stationary distribution of $\bar{Q}$ cannot have a $q$th moment.

If $-1+q+p^q= 0$, then again consider $V(x)=x^q$.  We have \begin{eqnarray*}\Delta V(x) &\sim  & p(x+1)^q\left(x+1-(x+2)\left(1-\frac{1}{x+1}\right)^q+\left(\frac{1+p(x-1)}{x+1}\right)^q\right) \\ & \sim & p(x+1)^q\left(x+1-(x+2)\left(1-\frac{1}{x+1}\right)^q+p^q\left(1+\frac{1/p-2}{x+1}\right)^q\right) \\ & = & p(x+1)^q\left(-1+q\frac{x+2}{x+1}-\frac{q(q-1)(x+2)}{2(x+1)^2}+p^q+p^q \frac{q(1/p-2)}{x+1}+O\left(x^{-2}\right)\right)\\ & \sim & p(x+1)^{q-1}\left(q+q(1-q)\left(\frac{1}{p}-\frac32\right)\right). \end{eqnarray*}  As $p<e^{-1}$, this again will be strictly positive for $x$ sufficiently large.
\end{proof}

Given the relationship between our quasi-stationary distribution $a$ and the stationary distribution $u$ of $\bar{Q}$, $a$ will have a $(q+1)$th moment if and only if $u$ has a $q$th moment.  Hence the criterion in Proposition \ref{qmoment} becomes $-3+\beta+p^{\beta-2}<0$ for $a$ to have a $(\beta-1)$th moment, so $a$ has tail behaviour close to that of a power law with index $\beta$ where $-3+\beta+p^{\beta-2}=0$ in the sense that it is lighter than any heavier tailed power law and heavier than any lighter tailed power law.  The second part of Proposition \ref{qmoment} and the $-1+q+p^q$ case of Proposition \ref{noqmoment} give stronger conditions on the tail.

\section{Convergence of conditional probabilities}\label{proofmain1}

In this section we complete the proof of parts (a) and (c) of Theorem \ref{main}.  We note that under the assumptions of the theorem Corollary \ref{prange} tells us that $\bar{Q}$ is positive recurrent and hence that Proposition \ref{converge} applies, meaning that the probability that in the continuous time model the tracked vertex $V_t$ has degree $d$ at time $t$, conditional on its degree being non-zero, is $a_d$.  It remains to prove that this also applies to a randomly chosen vertex.

We define a continuous time process $(U_t)_{t\geq 0}$ by, each time a vertex is added to the graph, moving to a vertex chosen uniformly at random from the vertices of the new graph.  This ensures that $\pr(U_t=v|N_t=n)=1/n$ for $v\leq n$.  Let $\bar{D}_t$ be the degree of $U_t$ in $\Gamma_t$.

\begin{lem}\label{randomvertex} Given $\epsilon>0$, there exists $v_{\epsilon}$ such that for $v\geq v_{\epsilon}$ $1\leq \frac{\pr(V_t=v)}{\pr(U_t=v)} \leq 1+\epsilon$\end{lem}

\begin{proof}
First of all we note that we can consider the tracking process in the discrete time model, letting $\tilde{V}_n$ be the tracked vertex at time $n$.  It is then easy to show by induction on $n$ that for any non-initial vertex $v>n_0$ we have $\pr(\tilde{V}_n=v)=\frac{1}{n-1}$ and that for an initial vertex $v \leq n_0$ we have $\pr(\tilde{V}_n=v)=\frac{n_0-1}{n_0(n-1)}$.

In the continuous time model, the sequence of changes of tracking is independent of the times of the duplication events, so we can conclude that $$\pr(V_t=v|N_t=n)=\left\{\begin{array}{lr} \frac{1}{n-1} & n_0<v\leq n \\ \frac{n_0-1}{n_0(n-1)} & v \leq n_0\end{array}\right.$$

The number of vertices at time $t$ is Negative Binomial with parameters $n_0$ and $e^{-t}$ (which can be deduced from Yule \cite{yule1925}), so, for $v>n_0$, $$\pr(V_t=v)=\sum_{n=v}^{\infty}\frac{1}{n-1}\binom{n}{n_0}e^{-n_0 t}(1-e^{-t})^{n-n_0}.$$  Similarly
$$\pr(U_t=v)=\sum_{n=v}^{\infty}\frac{1}{n}\binom{n}{n_0}e^{-n_0 t}(1-e^{-t})^{n-n_0}.$$
Hence, given $\epsilon>0$, there exists $v_{\epsilon}$ such that for $v\geq v_{\epsilon}$ $1\leq \frac{\pr(V_t=v)}{\pr(U_t=v)} \leq 1+\epsilon$.
\end{proof}

In the continuous time model, both the events that $V_t=v$ and $U_t=v$ are independent of the degree of $v$, and $\pr(V_t<v_{\epsilon})\to 0$ as $t\to\infty$.  Hence we have that $$\frac{\pr(D_t=k)}{\pr(\bar{D}_t=k)}\to 1$$ and $$\frac{\pr(D_t>0)}{\pr(\bar{D}_t>0)}\to 1$$ as $t\to\infty$, which completes the proof that $\pr(U_n=k|U_n\neq 0)=a_k$, and, to complete the proof of part (a) of Theorem \ref{main}, note that if conditioning on $G_m$ we can simply relabel $G_m$ as $\Gamma_0$. Part (c) then follows from Propositions \ref{qmoment} and \ref{noqmoment}.

\section{Convergence in probability}\label{proofmain2}

In this section we will complete the proof of part (b) of Theorem \ref{main}.  We will be working with the continuous time embedding $(\Gamma_t)_{t\geq 0}$, and for now we will assume that the initial graph $\Gamma_0$ is two vertices connected by a single edge, so that $n_0=2$.

Let $E_t$ be the number of edges of $\Gamma_t$ at time $t$; then $E_t=e^{2pt}W_t$ where $(W_t)_{t\geq 0}$ is a non-negative martingale, and by Theorem 2.9 of Hermann and Pfaffelhuber \cite{pfaffel} we know that $W_t$ converges in $L^2$ to a limit $W$.  We first show a slight strengthening of the part of Theorem 2.9 of \cite{pfaffel} which refers to the number of edges.

\begin{lem}\label{edgelimit}We have that $\pr(W=0)=0$.\end{lem}
\begin{proof}
Almost surely, there will be $s$ such that $\Gamma_s$ has two edges $i$ and $j$ which do not share a vertex.  For $t>0$, we can then consider the subgraphs of $\Gamma_{s+t}$, which we will refer to as $G_{t}^{(i)}$ and $G_{t}^{(j)}$, descended from the edges $i$ and $j$, and the fact that the edges do not share an endpoint means that these two graph processes are independent.  Let $E_{t}^{(i)}$ and $E_{t}^{(j)}$ be the numbers of edges in the two subgraphs and let $W_{t}^{(i)}$ and $W_{t}^{(j)}$ be the corresponding martingales, with limits $W^{(i)}$ and $W^{(j)}$.  Then $\pr(W=0)\leq \pr(W^{(i)}=0)\pr(W^{(j)}=0)$, so $\pr(W=0)$ is either $0$ or $1$, but it is shown in \cite{pfaffel} that $\pr(W=0)<1$.
\end{proof}

Consider the graph at time $s$, when it has $E_s=e^{2ps}W_s$ edges.  We use a similar idea as in the proof of Lemma \ref{edgelimit}, decomposing the graph $\Gamma_t$ for $t>s$ as a union of graphs $\Gamma_{t-s}^{(i)}$ descended from edge $i$ of $\Gamma_s$, which clearly then each have the same distribution as $\Gamma_{t-s}$.  Let the number of vertices of degree $k$ of $\Gamma_{t}^{(i)}$ at time $t$ be $N_{t,k}^{(i)}$.

We note that the processes $(\Gamma_{t}^{(i)})_{t\geq 0}$ and $(\Gamma_{t}^{(j)})_{t\geq 0}$ depend only on duplication events at the vertices of edges $i$ and $j$ and their descendants and so are independent if edges $i$ and $j$ do not have a vertex in common; furthermore we note that the number of pairs of edges which do have a vertex in common is given by the number of 2-stars $S_s$ in the graph $\Gamma_s$, which by the second part of Theorem 2.9 of \cite{pfaffel} we know is equal to $e^{(2p+p^2)s}\tilde{S}_s$ where $(\tilde{S}_t)\to S$ for some limiting random variable $S$.

For fixed $s$ and $t>s$, consider the random variable $$\hat{N}_{t,k}=\sum_{i=1}^{E_s}N_{t-s,k}^{(i)},$$ which can be thought of as the total number of degree $k$ vertices at time $t-s$ in all the subgraphs descended from each edge of the graph at time $s$ when considered separately.  Then $$\E(e^{-2pt}\hat{N}_{t,k}|\salj_s)=W_s\E(e^{-2p(t-s)}N_{t-s,k}),$$ and we have \begin{eqnarray*}\Var\left(e^{-2pt}\hat{N}_{t,k}| \salj_s \right) &=& e^{-4ps}\Var\left(\sum_{i=1}^{E_s}e^{-2p(t-s)}N_{t-s,k}^{(i)}\right) \\ &\leq & e^{-4ps}[W_se^{2ps}+2e^{(2p+p^2)s}\tilde{S}_s]\Var(e^{-2p(t-s)}N_{t-s,k}),\end{eqnarray*} and as $e^{-2pt}N_{t,k}<e^{-2pt}E_t=W_t$ we know that $\Var(e^{-2p(t-s)}N_{t-s,k})$ is bounded as $t\to\infty$.  Hence, for any $u>0$ $\Var\left(e^{-2p(s+u)}\hat{N}_{s+u,k}| \salj_s \right)\to 0$ as $s\to\infty$, and hence $$\Var(e^{-2p(s+u)}\hat{N}_{s+u,k}-W_s\E(e^{-2pu}N_{u,k}))\to 0.$$

As we assume that $\Gamma_0$ consists of two vertices connected by a single edge, the initial degrees are $1$, so using Lemma \ref{randomvertex} and the first part of Proposition \ref{converge}, we have that $e^{(1-2p)t}\pr(\tilde{D}_t=k)\to m_k$ as $t\to\infty$.  As the number of vertices at time $t$ is Negative Binomial with parameters $2$ and $e^{-t}$, $\E(N_{t,k})=2e^t\pr(\tilde{D}_t=k)$ and so $e^{-2pt}\E(N_{t,k}) \to 2m_k$ as $t\to\infty$.  Using $W_s\to W$ in $L^2$ and $\E(e^{-2pu}N_{u,k}) \to 2m_k$ as $u\to\infty$, we have $e^{-2pt}\hat{N}_{t,k}\to 2m_kW$ in $L^2$ as $t\to\infty$.

It remains to show that $\hat{N}_{t,k}$ is close to $N_{t,k}$.  To do this, consider a vertex $v$ in $\Gamma_s$ with degree $j$, and consider starting a tracked vertex process from this vertex.  As well as the Markov chain $D_t$ which starts from $j$ at time $s$ giving the degree of the tracked vertex, we can also consider Markov chains $D_t^{(i)}$ which start from $1$ and whose values are the degree of the tracked vertex in the subgraph descended from edge $i$, where $i$ is one of the edges incident on $v$. Then Proposition \ref{converge} shows that $$\lim_{t\to\infty} \pr(D_t \geq 1)e^{(1-2p)(t-s)}=jm_k,$$ and that $$\lim_{t\to\infty} \pr(D_t^{(i)} \geq 1)e^{(1-2p)(t-s)}=m_k.$$  Hence the probability that more than one of the $D_t^{(i)}$ is positive is $o(e^{(1-2p)t})$ as $t\to\infty$, and hence $\E|\hat{N}_{t,k}-N_{t,k}|=o(e^{2pt})$.

We can apply the same argument to the total number of vertices with positive degree at time $t$, $\sum_{k=1}^{\infty} N_{t,k}$, showing that it converges in $L^1$ to $2\sum_{k=1}^{\infty}m_k W$.  Hence we can conclude that $e^{-2pt}\E(N_{t,k})$ converges in $L^1$ to $2m_kW$ as $t\to\infty$ and that $e^{-2pt}\sum_{k=1}^{\infty}\E(N_{t,k})$ converges in $L^1$ to $2W\sum_{k=1}^{\infty}m_k$ as $t\to\infty$; hence the proportion of vertices in the connected component converges to $m_k/\sum_{k=1}^{\infty}m_k=a_k$ in probability as $t\to\infty$.

Finally, if we start with a more general graph $\Gamma_0$ we can apply the above argument to the subgraphs descended from each edge, and use the same idea as above to obtain the behaviour of the graph as a whole.  This completes the proof.

\bibliographystyle{plain}
\bibliography{dupgraphs}

\end{document}